\documentclass[12pt]{article}
\usepackage{a4wide}
\usepackage{amsthm}
\usepackage{amsfonts}
\usepackage{amssymb}
\usepackage{amsmath}
\usepackage{cite}
\usepackage{epsfig}

\newcommand\dd{\,\mbox{d}}

\newcommand\numv[1]{v\left(#1\right)}
\newcommand\nume[1]{e\left(#1\right)}

\newtheorem{theorem}{Theorem}

\newtheorem{lemma}[theorem]{Lemma}
\newtheorem{conjecture}{Conjecture}

\DeclareTextCompositeCommand{\v}{OT1}{l}{l\nobreak\hspace{-.1em}'}
\DeclareTextCompositeCommand{\v}{OT1}{t}{t\nobreak\hspace{-.1em}'\nobreak\hspace{-.15em}}

\begin{document}
\title{Ramsey multiplicity of apices of trees\thanks{The first and third authors were supported by the MUNI Award in Science and Humanities (MUNI/I/1677/2018) of the Grant Agency of Masaryk University. The second author was supported by Slovenian Research and Innovation Agency (P1-0383, J1-4008, J1-3002 and N1-0370).}}

\author{Daniel Kr{\'a}\v{l}\thanks{Institute of Mathematics, Leipzig University, Augustusplatz 10, 04109 Leipzig, and Max Planck Institute for Mathematics in the Sciences, Inselstra{\ss}e 22, 04103 Leipzig, Germany. E-mail: {\tt daniel.kral@uni-leipzig.de}. Previous affiliation: Faculty of Informatics, Masaryk University, Botanick\'a 68A, 602 00 Brno, Czech Republic.}\and
        Matja\v z Krnc\thanks{Faculty of Mathematics, Natural Sciences and Information Technologies, University of Primorska, Glagolja\v ska 8, SI-6000 Koper, Slovenia. E-mail:  {\tt matjaz.krnc@upr.si}.}\and
	Ander Lamaison\thanks{Institute for Basic Science, 55 Expo-ro, Yuseong-gu, 34126 Daejeon, South Korea. E-mail: {\tt ander@ibs.re.kr}. Previous affiliation: Faculty of Informatics, Masaryk University, Botanick\'a 68A, 602 00 Brno, Czech Republic.}
	}

\date{} 
\maketitle

\begin{abstract}
A graph $H$ is common
if its Ramsey multiplicity,
i.e., the minimum number of monochromatic copies of $H$ contained in any $2$-edge-coloring of $K_n$,
is asymptotically the same as the number of monochromatic copies in the random $2$-edge-coloring of $K_n$.
Erd\H os conjectured that every complete graph is common,
which was disproved by Thomason in the 1980s.
Till today, a classification of common graphs remains a widely open challenging problem.
Grzesik, Lee, Lidick\'y and Volec [Combin. Prob. Comput. 31 (2022), 907--923] conjectured that
every $k$-apex of any connected Sidorenko graph is common.
We prove for $k\le 5$ that the $k$-apex of any tree is common.
\end{abstract}

\section{Introduction}
\label{sec:intro}

Ramsey's Theorem~\cite{Ram30} asserts that for every graph $H$,
every $2$-edge-coloring of a sufficiently large complete graph contains a monochromatic copy of $H$.
In this short paper, we are interested in the quantitative version of the statement:
What is the minimum number of monochromatic copies of $H$ in a $2$-edge-coloring of $K_n$?
The \emph{Ramsey multiplicity} of a graph $H$, which is denoted by $M(H,n)$,
is the minimum number of monochromatic (labeled) copies of $H$ in a $2$-edge-coloring of $K_n$, and
the \emph{Ramsey multiplicity constant} of $H$, denoted by $C(H)$,
is the limit
\[C(H)=\lim_{n\to\infty}\frac{M(H,n)}{n^{\numv{H}}}.\]
The random $2$-edge-coloring yields that $C(H)\le 2^{1-\nume{H}}$ for every graph $H$, and
those graphs $H$ such that $C(H)=2^{1-\nume{H}}$ are called \emph{common}.
For example, Goodman's Theorem~\cite{Goo59} implies that $K_3$ is common.
Since every graph that has the Sidorenko property is common,
the famous conjecture of Sidorenko~\cite{Sid93} and of Erd\H{o}s and Simonovits~\cite{ErdS83},
which asserts that every bipartite graph has the Sidorenko property (the definition is given in Section~\ref{sec:notation}),
would imply, if true, that every bipartite graph is common.

The notion of common graphs can be traced to the 1960s
when Erd\H{o}s~\cite{Erd62} conjectured that every complete graph is common.
Later, Burr and Rosta~\cite{BurR80} extended the conjecture to all graphs.
Both conjectures turned out to be false:
Sidorenko~\cite{Sid86,Sid89} showed that a triangle with a pendant edge is not common and
Thomason~\cite{Tho89} showed that $K_4$ is not common.
More generally, any graph containing $K_4$ is not common~\cite{JagST96}, and
any non-bipartite graph can be made to be not common by adding a sufficiently large tree~\cite{Fox08}.

On the positive side, every graph with the Sidorenko property is common.
In particular, families of bipartite graphs known
to have the Sidorenko property~\cite{BlaR65, Sid89, Sid91, ConFS10, ConL17, ConKLL18, ConL21}
provide examples of bipartite graphs that are common.
Examples of non-bipartite common graphs include odd cycles~\cite{Sid89} and even wheels~\cite{JagST96,Sid96}, and
more generally graphs that can be recursively obtained by gluing copies of $K_3$ by a vertex or an edge~\cite{GrzLLV22}.
The existence of a common graph with chromatic number at least four was open until 2012
when the flag algebra method of Razborov~\cite{Raz07} resulted in a proof that
the $5$-wheel is common~\cite{HatHKNR12}.
The existence of common graphs with arbitrarily large chromatic number was established only recently~\cite{KraVW25};
also see~\cite{KL23} for the construction of such highly connected graphs derived from those presented in~\cite{KraVW25}.

A characterization of common graphs remains an intriguing open problem, which is very far from our reach.
Our work is motivated by the following conjecture of Grzesik, Lee, Lidick\'y and Volec~\cite{GrzLLV22}.
Recall that the \emph{$k$-apex} of a graph $G$ is the graph obtained from $G$ by adding $k$ vertices and
making each of the new $k$ vertices adjacent to every vertex of $G$ (however, the $k$ new vertices
remain to form an independent set).

\begin{conjecture}[{Grzesik, Lee, Lidick\'y and Volec~\cite[Conjecture~1.3]{GrzLLV22}}]
\label{conj:main}
If $G$ is a connected graph that has the Sidorenko property,
then the $k$-apex of $G$ is common for every $k$.
\end{conjecture}

Recall that the above mentioned conjecture of Sidorenko~\cite{Sid93} and of Erd\H{o}s and Simonovits~\cite{ErdS83}
asserts that every bipartite graph has the Sidorenko property,
i.e., if both conjectures were true, then every $k$-apex of every bipartite graph would be common.
The assumption in Conjecture~\ref{conj:main} that the graph $G$ is connected is necessary:
the $1$-apex of the union of $K_1$ and $K_2$ is a triangle with a pendant edge,
which is not common~\cite{Sid86,Sid89}.

Conjecture~\ref{conj:main} is known for a small number of special classes of graphs:
$1$-apices of even cycles, which are even wheels~\cite{JagST96,Sid96},
$1$-apices of trees, which are graphs obtained by gluing triangles on edges~\cite{JagST96,Sid96},
$2$-apices of even cycles~\cite{GrzLLV22}, and
$k$-apices of connected bipartite graphs with at most five vertices~\cite{GrzLLV22}.
Our main result is a proof of the conjecture for $k$-apices of trees when $k\le 5$.

\begin{theorem}
\label{thm:main}
Let $T$ be a tree. The $k$-apex of $T$ for $k\in\{1,2,3,4,5\}$ is common.
\end{theorem}

Note that Theorem~\ref{thm:main} is the first example of an infinite class of bipartite graphs such that
$k$-apices of the graphs from the class are common for $k\ge 3$.
We also remark that, unlike some results in the area, the proof of Theorem~\ref{thm:main} is not computer assisted.

This paper is structured as follows.
After reviewing notation in Section~\ref{sec:notation},
we prove Theorem~\ref{thm:main} in Section~\ref{sec:main}.
Interestingly,
we prove a stronger statement that
concerns all $n$-vertex graphs with $n-1$ edges that have the Sidorenko property.
In Section~\ref{sec:concl},
we discuss the possibility of extending our arguments to $k$-apices of trees for $k>5$ and
present an example that shows that the inequality claimed in Lemma~\ref{lm:main},
which is essential in our argument, fails for $k\ge 12$.

\section{Notation}
\label{sec:notation}

We now review notation used throughout the paper.
All graphs considered in this paper are finite and simple.
If $G$ is a graph,
then $V(G)$ denotes the vertex set of $G$ and $E(G)$ denotes the edge set of $G$.
The number of vertices of a graph $G$ is denoted by $\numv{G}$ and the number of edges by $\nume{G}$.
Our arguments often involve the following two classes of graphs and
so we introduce particular notation for them (see Figure~\ref{fig:SkTk} for illustration):
$S_k$ denotes a \emph{star} with $k$ leaves, and
$T_k$ denotes a \emph{triangular book} with $k$ triangles (the complete tripartite $K_{1,1,k}$).
Finally, the $k$-apex of a graph $G$, which is denoted by $G^{+k}$,
is the graph obtained from $G$ by adding $k$ new vertices and
making each of the new $k$ vertices adjacent to every vertex of $G$ (the $k$ newly added vertices remain to form an independent set).
In particular, $S_k$ is the $k$-apex of $K_1$ and $T_k$ is the $k$-apex of $K_2$.

\begin{figure}
\begin{center}
\epsfbox{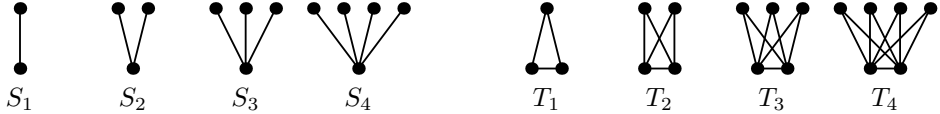}
\end{center}
\caption{The graphs $S_k$ and $T_k$ for $k\in\{1,2,3,4\}$.}
\label{fig:SkTk}
\end{figure}

We treat the notion of common graphs using tools from the theory of graph limits;
we refer to the monograph by Lov\'asz~\cite{Lov12} for a comprehensive treatment of graph limits.
In the theory of graph limits, large graphs are represented by an analytic object called graphon.
A~\emph{graphon} is a symmetric measurable function $W:[0,1]^2\to [0,1]$,
i.e., $W(x,y)=W(y,x)$ for all $x,y\in [0,1]$.
Since a graphon can be thought of as a continuous version of the adjacency matrix of a graph (although
a more precise view is as the limit of regularity partitions that approximate large graphs),
we will refer to the elements of the domain $[0,1]$ as vertices of $W$.
The \emph{homomorphism density} of a graph $G$ in a graphon $W$ is defined as
\[
t(G,W) = \int\limits_{[0,1]^{V(G)}} \prod\limits_{uv\in E(G)} W(x_u,x_v) \dd x_{V(G)}.
\]
We say that a graph $G$ has the \emph{Sidorenko property}
if $t(G,W)\ge t(K_2,W)^{e(G)}$ for every graphon $W$.
Recall that
the aforementioned conjecture of Sidorenko~\cite{Sid93} and of Erd\H{o}s and Simonovits~\cite{ErdS83}
asserts that
every bipartite graph has the Sidorenko property.
It is interesting to note that a graph $G$ has the Sidorenko property if and only if
$G$ is $k$-common for every $k$~\cite{KraNNVW22},
i.e., the random $k$-edge-coloring of a complete graph asymptotically minimizes the number of monochromatic copies of $G$
among all $k$-edge-colorings of the host complete graph.

We will often think of a graphon $W$ as representing one color class of edges
while the complementary color class is represented by the graphon $1-W$,
which is defined as $(1-W)(x,y)=1-W(x,y)$.
With this view in mind, we define the \emph{monochromatic density} of a graph $G$ in a graphon $W$ as
\[
m(G,W)=t(G,W)+t(G,1-W).
\]
The definition of being common can now be cast as follows:
a graph $G$ is common if and only if $m(G,W)\ge 2^{1-\nume{G}}$ for every graphon $W$.

In our arguments, we often consider monochromatic common neighborhoods of $k$ vertices of a graphon.
To make our presentation easier,
we set $D_k$ as $\{0,1\}\times [0,1]^k$ and
we think of the first coordinate as representing the color class and
the remaining $k$ coordinates as the $k$-tuple of vertices of a graphon.
We think of $D_k$ as equipped with a measure $\mu$ defined as
\[\mu\left(\left(\{0\}\times X\right)\cup\left(\{1\}\times Y\right)\right)=\lambda(X)+\lambda(Y)\]
for $X,Y\subseteq [0,1]^k$ where $\lambda$ is the uniform measure on $[0,1]^k$.
Using the above mentioned interpretation of the elements of $D_k$,
we define the function $d^W_k:D_k\to [0,1]$ that
represents the monochromatic common degree (the size of monochromatic common neighborhood) of a $k$-tuple of vertices as follows:
\[
d^W_k(s,x_1,\ldots,x_k)=
  \begin{cases}
  \int\limits_{[0,1]}\prod\limits_{i=1}^k W(x_i,y)\dd y
  & \text{if }s=0 \\
  \int\limits_{[0,1]}\prod\limits_{i=1}^k (1-W(x_i,y))\dd y
  & \text{if }s=1.
  \end{cases}
\]
Observe that
\[m(S_k,W)=\int\limits_{D_k}d^W_k(z)\dd z\]
for every $k$ (recall that $D_k=\{0,1\}\times [0,1]^k$).
Similarly,
we define the function $\delta^W_k:D_k\to [0,1]$ that
represents the edge density of the monochromatic common neighborhood of a $k$-tuple of vertices as follows (informally speaking,
the numerator counts the number of pairs of common neighbors that are adjacent and
the denominator counts the number of pairs of common neighbors):
\[
\delta^W_k(s,x_1,\ldots,x_k)=
\begin{cases}
\frac{\int\limits_{[0,1]^2} W(y,y') \prod\limits_{i=1}^k W(x_i,y)W(x_i,y')\dd y\dd y'}{d^W_k(0,x_1,\ldots,x_k)^2}
& \text{if }s=0 \\
&\\
\frac{\int\limits_{[0,1]^2} (1-W(y,y')) \prod\limits_{i=1}^k (1-W(x_i,y))(1-W(x_i,y'))\dd y\dd y'}{d^W_k(1,x_1,\ldots,x_k)^2}
& \text{if }s=1.
\end{cases}
\]
Observe that
\[m(T_k,W)=\int\limits_{D_k}d^W_k(z)^2\delta^W_k(z)\dd z,\]
and more generally,
if a graph $G$ has the Sidorenko property,
it holds that
\begin{equation}
m(G^{+k},W)\ge\int\limits_{D_k}d^W_k(z)^{\numv{G}}\delta^W_k(z)^{\nume{G}}\dd z
\label{eq:Sidorenko0}
\end{equation}
for every $k$.

\section{Main result}
\label{sec:main}

In this section, we present the proof of our main result.
The following lemma is the core of our argument.
In Section~\ref{sec:concl},
we present an example that the statement of the lemma fails for $k\ge 12$.

\begin{lemma}
\label{lm:main}
For every $k\in\{1,\ldots,5\}$ and every graphon $W$, it holds that
\begin{equation}
\frac{m(T_k,W)}{m(S_k,W)}\ge 2^{-(k+1)}.\label{eq:TS}
\end{equation}
\end{lemma}

\begin{proof}
Fix a graphon $W$.
We define a function $h:[0,1]\to [-1/2,+1/2]$,
which measures how much the degree of a particular vertex $x$ of $W$ differs from $1/2$, as follows:
\[h(x)=\int\limits_{[0,1]}W(x,y)\dd y-\frac{1}{2}.\]
Observe that
\begin{equation}
m(S_k,W)=\int\limits_{[0,1]}\left(\frac{1}{2}+h(x)\right)^k+\left(\frac{1}{2}-h(x)\right)^k\dd x\label{eq:Sk}
\end{equation}
for every positive integer $k$.
Goodman's formula~\cite{Goo59} asserts that
\[m(T_1,W)=\frac{3}{2}m(S_2,W)-\frac{1}{2},\]
which yields that
\[m(T_1,W)=\frac{3}{2}\int\limits_{[0,1]}\left(\frac{1}{2}+h(x)\right)^2+\left(\frac{1}{2}-h(x)\right)^2\dd x-\frac{1}{2}
          =\frac{1}{4}+3\int\limits_{[0,1]}h(x)^2\dd x.\]
The well-known inequality $m(T_k,W)\ge m(T_1,W)^k$,
which follows from a standard H\"older type argument (the inequality is also a particular case of~\cite[Lemma 2.2]{GrzLLV22}),
implies that
\begin{equation}
m(T_k,W)\ge\left(\frac{1}{4}+3\int\limits_{[0,1]}h(x)^2\dd x\right)^k
        \ge\frac{1}{2^{2k}}+\frac{3k}{2^{2k-2}}\int\limits_{[0,1]}h(x)^2\dd x
\label{eq:Tk}
\end{equation}
for every positive integer $k$.
Note that \eqref{eq:Tk} implies that $m(T_k,W)\ge 2^{-2k}$, i.e., $T_k$ is common.

If $k=1$, we derive the desired estimate from \eqref{eq:Sk} and \eqref{eq:Tk} as follows:
\[\frac{m(T_1,W)}{m(S_1,W)}\ge
  \frac{\frac{1}{4}+3\int\limits_{[0,1]}h(x)^2\dd x}
       {\int\limits_{[0,1]}\left(\frac{1}{2}+h(x)\right)+\left(\frac{1}{2}-h(x)\right)\dd x}=
  \frac{\frac{1}{4}+3\int\limits_{[0,1]}h(x)^2\dd x}{1}\ge\frac{1}{4}.\]
If $k=2$, we again derive the desired estimate from \eqref{eq:Sk} and \eqref{eq:Tk}:
\[\frac{m(T_2,W)}{m(S_2,W)}\ge
  \frac{\frac{1}{16}+\frac{6}{4}\int\limits_{[0,1]}h(x)^2\dd x}
       {\int\limits_{[0,1]}\left(\frac{1}{2}+h(x)\right)^2+\left(\frac{1}{2}-h(x)\right)^2\dd x}=
  \frac{\frac{1}{16}+\frac{24}{16}\int\limits_{[0,1]}h(x)^2\dd x}
       {\frac{1}{2}+\frac{4}{2}\int\limits_{[0,1]}h(x)^2\dd x}\ge\frac{1}{8}.\]
And if $k=3$, we derive the desired estimate from \eqref{eq:Sk} and \eqref{eq:Tk} as follows:
\[\frac{m(T_3,W)}{m(S_3,W)}\ge
  \frac{\frac{1}{64}+\frac{9}{16}\int\limits_{[0,1]}h(x)^2\dd x}
       {\int\limits_{[0,1]}\left(\frac{1}{2}+h(x)\right)^3+\left(\frac{1}{2}-h(x)\right)^3\dd x}=
  \frac{\frac{1}{64}+\frac{36}{64}\int\limits_{[0,1]}h(x)^2\dd x}
       {\frac{1}{4}+\frac{12}{4}\int\limits_{[0,1]}h(x)^2\dd x}\ge\frac{1}{16}.\]
For $k\in\{4,5\}$, we employ the following estimate,
which holds as $\lvert h(x)\rvert\le 1/2$:
\begin{equation}
\int\limits_{[0,1]}h(x)^4\dd x \le \frac{1}{4}\int\limits_{[0,1]}h(x)^2\dd x\label{eq:h4}
\end{equation}
If $k=4$, we combine \eqref{eq:Sk}, \eqref{eq:Tk} and \eqref{eq:h4} as follows:
\[\frac{m(T_4,W)}{m(S_4,W)}\ge
  \frac{\frac{1}{256}+\frac{12}{64}\int\limits_{[0,1]}h(x)^2\dd x}
       {\int\limits_{[0,1]}\left(\frac{1}{2}+h(x)\right)^4+\left(\frac{1}{2}-h(x)\right)^4\dd x}\ge
  \frac{\frac{1}{256}+\frac{48}{256}\int\limits_{[0,1]}h(x)^2\dd x}
       {\frac{1}{8}+\frac{24+4}{8}\int\limits_{[0,1]}h(x)^2\dd x}\ge\frac{1}{32}.\]
Finally,
if $k=5$, we obtain using \eqref{eq:Sk}, \eqref{eq:Tk} and \eqref{eq:h4} the following:
\[\frac{m(T_5,W)}{m(S_5,W)}\ge
  \frac{\frac{1}{1024}+\frac{15}{256}\int\limits_{[0,1]}h(x)^2\dd x}
       {\int\limits_{[0,1]}\left(\frac{1}{2}+h(x)\right)^5+\left(\frac{1}{2}-h(x)\right)^5\dd x}\ge
  \frac{\frac{1}{1024}+\frac{60}{1024}\int\limits_{[0,1]}h(x)^2\dd x}
       {\frac{1}{16}+\frac{40+20}{16}\int\limits_{[0,1]}h(x)^2\dd x}=\frac{1}{64}.\]
The proof of the lemma is now completed.
\end{proof}

We are now ready to prove our main result, which is implied by the following more general theorem.

\begin{theorem}
\label{thm:main2}
Let $G$ be an $n$-vertex graph with $n-1$ edges that has the Sidorenko property.
The $k$-apex $G^{+k}$ of $G$ for $k\in\{1,2,3,4,5\}$ is common.
\end{theorem}

\begin{proof}
Fix an $n$-vertex graph $G$ with $n-1$ edges that has the Sidorenko property,
an integer $k\in\{1,\ldots,5\}$, and a graphon $W$.
We need to show that
\begin{equation}
m(G^{+k},W)\ge 2^{1-\nume{G^{+k}}}=2^{1-k\cdot\numv{G}-\nume{G}}=2^{2-(k+1)n}.
\label{eq:goal}
\end{equation}

Since the graph $G$ has the Sidorenko property,
we derive from \eqref{eq:Sidorenko0} that
\begin{equation}
m(G^{+k},W)\ge\int\limits_{D_k}d^W_k(z)^{n}\delta^W_k(z)^{n-1}\dd z.
\label{eq:Sidorenko}
\end{equation}
H\"older's Inequality imply that
\begin{equation}
\int\limits_{D_k}d^W_k(z)^2\delta^W_k(z)\dd z\le \left(\;\int\limits_{D_k}d^W_k(z)\dd z\right)^{\frac{n-2}{n-1}}
                                 \left(\;\int\limits_{D_k}d^W_k(z)^{n}\delta^W_k(z)^{n-1}\right)^{\frac{1}{n-1}}.
\label{eq:Holder}
\end{equation}
We next combine \eqref{eq:Sidorenko} and \eqref{eq:Holder} to obtain that
\begin{equation}
m(G^{+k},W) \ge \frac{\left(\int\limits_{D_k}d^W_k(z)^2\delta^W_k(z)\dd z\right)^{n-1}}{\left(\int\limits_{D_k}d^W_k(z)\dd z\right)^{n-2}}
            = \frac{m(T_k,W)^{n-1}}{m(S_k,W)^{n-2}} 
\label{eq:almost}
\end{equation}
Since the graph $T_k$ is common~\cite{GrzLLV22} (this also follows from \eqref{eq:Tk} as mentioned earlier),
it holds that $m(T_k,W)\ge 2^{-2k}$.
We now use Lemma~\ref{lm:main} and derive from \eqref{eq:almost} that
\[
m(G^{+k},W) \ge m(T_k,W)\left(\frac{m(T_k,W)}{m(S_k,W)}\right)^{n-2}
            \ge 2^{-2k}\cdot 2^{-(n-2)(k+1)}=2^{2-(k+1)n},
\]
which is the estimate \eqref{eq:goal} required to be established.
\end{proof}

\section{Concluding remarks}
\label{sec:concl}

We would like to conclude with two brief remarks on our proof.
We find it interesting that Theorem~\ref{thm:main} follows from a more general Theorem~\ref{thm:main2},
which does not require the graph $G$ to be connected.
We speculate that the assumption on connectivity of a graph $G$ in Conjecture~\ref{conj:main}
can be replaced with a lower bound on the number of edges of $G$.
Theorem~\ref{thm:main2} also implies for $k\in\{1,2,3,4,5\}$ that
the $k$-apex of any Sidorenko graph is an induced subgraph of a common graph.
In particular,
the $1$-apex of any Sidorenko graph with a suitable number pendant vertices adjacent to the apex vertex is common,
however, when a sufficiently large number of such pendant vertices is added,
the graph is not common by the result of Fox~\cite{Fox08} mentioned in Introduction.

One can wonder whether the statement of Lemma~\ref{lm:main} can hold for any $k$ (note that
if the statement of the lemma holds for $k$,
then the statement of Theorem~\ref{thm:main2} holds for the same value of $k$).
Unfortunately, this is not true as shown by the following example.
For $\varepsilon\in [0,1)$ consider the graphon $W_\varepsilon$ constructed as follows.
Split the interval $[0,1]$ into an interval $A$ of length $\varepsilon$ and
three intervals $B$, $C$ and $D$, each of length $\frac{1-\varepsilon}{3}$.
The graphon $W_\varepsilon$ is equal to one on the sets $(A\cup B)\setminus A^2$,
$(A\cup C)\setminus A^2$ and $(A\cup D)\setminus A^2$, and equal to zero elsewhere.
The construction is illustrated in Figure~\ref{fig:Weps}.

\begin{figure}
\begin{center}
\epsfbox{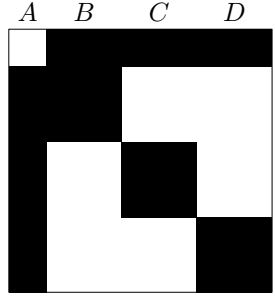}
\end{center}
\caption{The graphon $W_{\varepsilon}$ for $\varepsilon=1/7$.
         The origin of the coordinate system is in the top left corner,
	 the black color represents the value~$1$, and
	 the white color represents the value~$0$.}
\label{fig:Weps}
\end{figure}

A direct computation yields the following:
\begin{align*}
t(T_k,W_\varepsilon) & = 2\varepsilon(1-\varepsilon)\left(\frac{1-\varepsilon}{3}\right)^k + 3 \left(\frac{1-\varepsilon}{3}\right)^2\left(\frac{1+2\varepsilon}{3}\right)^k\\
t(T_k,1-W_\varepsilon) & = \varepsilon^{k+2} + 6 \left(\frac{1-\varepsilon}{3}\right)^{k+2}\\
t(S_k,W_\varepsilon) & = \varepsilon(1-\varepsilon)^k+(1-\varepsilon)\left(\frac{1+2\varepsilon}{3}\right)^k\\
t(S_k,1-W_\varepsilon) & = \varepsilon^{k+1}+(1-\varepsilon)\left(\frac{2-2\varepsilon}{3}\right)^k
\end{align*}
Observe that if $\varepsilon<2/5$, then
\[\lim_{k\to\infty}\frac{1}{k}\log\frac{m(T_k,W_\varepsilon)}{m(S_k,W_\varepsilon)}=
  \log\frac{\max\left\{\frac{1-\varepsilon}{3},\frac{1+2\varepsilon}{3},\varepsilon,\frac{1-\varepsilon}{3}\right\}}{\max\left\{1-\varepsilon,\frac{1+2\varepsilon}{3},\varepsilon,\frac{2-2\varepsilon}{3}\right\}}=
  \log\frac{\frac{1+2\varepsilon}{3}}{1-\varepsilon}=
  \log\frac{1+2\varepsilon}{3(1-\varepsilon)}.\]
In particular, the limit is equal to $\log\frac{10}{21}$ for $\varepsilon=1/8$,
however,
if Lemma~\ref{lm:main} were true for any value of $k$, this limit would be at least $\log\frac{1}{2}$.
In fact, this example shows that 
the statement Lemma~\ref{lm:main} fails already for $k=12$ as
$m(T_{12},W_{1/20})\approx 2.4849\cdot 10^{-6}$, $m(S_{12},W_{1/20})\approx 0.030980$ and
so \[\frac{m(T_{12},W_{1/20})}{m(S_{12},W_{1/20})}\approx 0.00008021\] while $2^{-13}\approx 0.0001221$.

\section*{Acknowledgement}

The authors would like to thank Andrzej Grzesik and Jan Volec
for fruitful discussions concerning Ramsey multiplicity of graphs including the results presented in this paper.

\bibliographystyle{bibstyle}
\bibliography{treeapex}

\end{document}